\documentclass[11pt]{article}

\usepackage[english]{babel}
\usepackage{amssymb}
\usepackage{amsmath}
\usepackage{hyperref}
\usepackage{enumerate}
\usepackage{amsthm}
\usepackage{amsfonts}
\usepackage{listings}
\usepackage[curve]{xypic}

\usepackage{MnSymbol}

\DeclareMathSymbol{\mlq}{\mathord}{operators}{``}
\DeclareMathSymbol{\mrq}{\mathord}{operators}{`'}

\title{Enlargement of (fibered) derivators}
\date{June 29, 2017}
\author{Fritz H\"ormann\\ Mathematisches Institut, Albert-Ludwigs-Universit\"at Freiburg}

\usepackage[numbers,sort&compress]{natbib}

\usepackage{color}
\usepackage[headsepline,footsepline]{scrpage2}

\usepackage{tikz}
\newcommand*\numcirc[1]{\tikz[baseline=(char.base)]{
            \node[shape=circle,draw,inner sep=2pt] (char) {#1};}}

\newtheorem{SATZ}{Theorem}[section]

\newtheorem{LEMMA}[SATZ]{Lemma}
\newtheorem{DEF}[SATZ]{Definition}
\newtheorem{PROP}[SATZ]{Proposition}

\newtheorem{KOR}[SATZ]{Corollary}

\newtheorem{BEM}[SATZ]{Remark}

\newtheoremstyle{bare}        % name
  {}            % Space above, empty = `usual value'
  {}            % Space below
  {\normalfont}                 % Body font (\normalfont)
  {}                            % Indent amount (empty = no indent, \parindent = para indent)
  {\bfseries}                   % Thm head font
  {}                            % Punctuation after thm head
  {.0em}                           % Space after thm head: " " = normal interword space;
  {\thmnumber{#2}#1{\thmnote{ \normalfont(#3)}}. } % Thm head spec \normalfont\textsc{}}

\theoremstyle{bare}
\newtheorem{PAR}[SATZ]{}

\newcommand{\commentempty}[1]{}

\newcommand{\iso}{\stackrel{\sim}{\longrightarrow}}

\newcommand{\DD}{ \mathbb{D} }
\newcommand{\EE}{ \mathbb{E} }

\newcommand{\SSS}{ \mathbb{S} }

\DeclareMathOperator{\id}{id}

\DeclareMathOperator{\op}{op}
\DeclareMathOperator{\cart}{cart}
\DeclareMathOperator{\cocart}{cocart}

\DeclareMathOperator{\Hom}{Hom}
\DeclareMathOperator{\Fun}{Fun}

\DeclareMathOperator{\pr}{pr}

\DeclareMathOperator{\cor}{cor}

\DeclareMathOperator{\Cat}{Cat}
\DeclareMathOperator{\Dia}{Dia}
\DeclareMathOperator{\Dir}{Dir}
\DeclareMathOperator{\Inv}{Inv}

\DeclareMathOperator{\Pos}{Pos}
\DeclareMathOperator{\Posf}{Posf}
\DeclareMathOperator{\Dirpos}{Dirpos}
\DeclareMathOperator{\Invpos}{Invpos}

\begin{document}

\maketitle



\section*{Abstract}

We show that the theory of derivators (or, more generally, of fibered multiderivators) on {\em all small categories} is equivalent to this theory on {\em partially ordered sets}, in the following sense: Every derivator (more generally, every fibered multiderivator) defined on partially ordered sets has an enlargement to all small categories that is unique up to equivalence of derivators.  
Furthermore, extending a theorem of Cisinski, we show that every bifibration of multi-model categories (basically a collection of model categories, and Quillen adjunctions in several variables between them) gives rise to a left and right fibered multiderivator on all small categories. 
%Along the lines it is also shown that every left derivator (or, more generally, every left fibered multiderivator) defined on 
%{\em directed posets} can be {\em intrinsically} enlarged to a left derivator (respectively to a left fibered multiderivator) on all small categories, completely independently of its provenience. Similarly for right derivators.
%\tableofcontents

\section{Introduction}

Let $\mathcal{M}$ be a model category. Cisinski has shown in \cite{Cis03} that the pre-derivator associated with $\mathcal{M}$, defined {\em on all small categories},  is a left and right derivator. This does {\em not} use any additional properties of $\mathcal{M}$ such as being combinatorial, or left or right proper. 
In this article we show that the analogous statement holds true also for fibered multiderivators (in particular for fibered derivators and monoidal derivators). More precisely:
Recall  \cite[Definition 4.1.3]{Hor15} that a bifibration of multi-model categories is a bifibration $\mathcal{D} \rightarrow \mathcal{S}$ of multicategories together with the 
choice of model category structures on the fibers such that the push-forward and pull-back functors along any multimorphism in $\mathcal{S}$ form a Quillen adjunction in $n$ variables (and an additional condition concerning ``units'', i.e.\@ 0-ary push-forwards). 

\begin{SATZ}[Theorem~\ref{SATZCISINSKI}]
Let $\mathcal{D} \rightarrow \mathcal{S}$ be a bifibration of multi-model categories, and let $\SSS$ be the represented pre-multiderivator of $\mathcal{S}$. 
For each small category $I$ denote by $\mathcal{W}_I$ the class of morphisms in $\Fun(I, \mathcal{D})$ which point-wise are weak equivalences in some fiber $\mathcal{D}_S$ (cf.\@ \cite[Definition 4.1.2]{Hor15}). 
The association
\[ I \mapsto \DD(I) := \Fun(I, \mathcal{D})[\mathcal{W}^{-1}_I]  \]
defines a left and right fibered multiderivator over $\SSS$ with domain $\Cat$\footnote{Small categories}. 
Furthermore the categories $\DD(I)$ are locally small. 
\end{SATZ}

Along the lines we also prove that a left (or right) fibered multiderivator defined on a smaller diagram category can be {\em intrinsically} enlarged to a larger diagram category, whenever some nerve-like construction is available, relating the two diagram categories. This is more of theoretical interest because most derivators occurring in nature come from model categories. One concrete result of the construction is the following:

\begin{KOR}
Let $\DD \rightarrow \SSS$ be a left (resp.\@ right) fibered multiderivator with domain $\Invpos$\footnote{Inverse posets} (resp.\@ $\Dirpos$\footnote{Directed posets}) 
such that $\SSS$ is defined on $\Cat$ and such that also (FDer0 right) (resp.\@ (FDer0 left)) holds.
Then there exists an enlargement of $\DD$ to a left (resp.\@ right) fibered multiderivator $\EE \rightarrow \SSS$ with domain $\Cat$, such that its restriction to $\Invpos$ (resp.\@ $\Dirpos$)
is equivalent to $\DD$. Any other such enlargement is equivalent to $\EE$. 
\end{KOR}
Note that this holds, in particular, for usual derivators (take for $\SSS$ the final pre-derivator) and closed monoidal derivators (take for $\SSS$ the final pre-multiderivator).
\begin{proof}
In the left case, apply the machine of Theorem~\ref{MAINTHEOREM} twice using the functors $N$ constructed in Proposition~\ref{PROPCONSTRN}, firstly for the pair $(\Invpos\subset \Cat^\circ)$, and secondly for
the pair $(\Inv \subset \Cat)$. Similarly for the right case.
\end{proof}

If we start with a fibered multiderivator on {\em all of $\Pos$}, however, we show that
the two extensions to $\Cat$ agree. Therefore we arrive at the following 
\begin{KOR}
Let $\DD \rightarrow \SSS$ be a left and right fibered multiderivator with domain $\Pos$ such that $\SSS$ is defined on $\Cat$.
Then there exists a canonical enlargement of $\DD$ to a left and right fibered multiderivator $\EE \rightarrow \SSS$ with domain $\Cat$, such that its restriction to $\Pos$
is equivalent to $\DD$. The enlargement is unique up to equivalence of fibered multiderivators over $\SSS$. 
\end{KOR}
Actually, here $\Pos$ can be even replaced by the smallest diagram category containing both $\Invpos$ and $\Dirpos$. 
\begin{proof}
We consider this time the pairs $(\Pos \subset \Cat^\circ)$ and $(\Cat^\circ \subset \Cat)$. For each of these pairs we dispose of functors $N$ as in \ref{PARAXIOMS} for the left and the right case simultaneously by Proposition~\ref{PROPCONSTRN} (by enlarging $\Dia'$ we only weaken the axioms). Hence we may conclude by applying Proposition~\ref{PROPLEFTRIGHT} twice.  
\end{proof}

For the reader mainly interested in plain (left and right) derivators, we state explicitly:
\begin{KOR}\label{KORDER}
Let $\DD$ be a derivator with domain $\Pos$. Then there exists a canonical enlargement of $\DD$ to a derivator $\EE$ with domain $\Cat$, such that its restriction to $\Pos$ is equivalent to $\DD$. The enlargement is unique up to equivalence of derivators. 
\end{KOR}

Thanks to Falk Beckert for pointing out that a weaker statement in the direction of Corollary~\ref{KORDER} has been proven by Jan Willing \cite{Wil95} in a diploma thesis under the supervision of Jens Franke. There, only stable derivators were considered under the name ``verfeinerte triangulierte Diagrammkategorien''.

\section{Relating different diagram categories via nerve-like constructions}

Let $\Dia$ be a diagram category, cf.\@ \cite[Definition 1.1.1]{Hor15}. In contrast to Axiom~(Dia3) of [loc.\@ cit.], in this article we require that $\Dia$ permits the construction 
of comma categories $I \times_{/J} K$ for arbitrary functors $\alpha: I \rightarrow J$ and $\beta: K \rightarrow J$ in $\Dia$. 
We assume that the reader is, to some extent, familiar with the definition of {\em fibered multiderivator} \cite[Section 1.2--3]{Hor15}.
The reader mainly interested in usual derivators or monoidal derivators can let $\SSS$ be the final pre-derivator (resp.\@ the final pre-multiderivator) and
then a ``left (resp.\@ right) fibered multiderivator over $\SSS$'' is just a left (resp.\@ right) derivator (resp.\@ a monoidal left (resp.\@ closed right) derivator). 

Recall the following \cite[Definition 2.4.1]{Hor15}:
\begin{DEF} Let $\DD \rightarrow \SSS$ be a (left or right) fibered derivator with domain $\Dia$. 
Let $I, E \in \Cat$ be diagrams with $I \in \Dia$ and let $\pi: I \rightarrow E$ be a functor. 
We say that an object \[ X \in \DD(I)\] is $\pi$-{\bf (co)Cartesian}, if for any morphism
$\mu: i \rightarrow j$ in $I$ mapping to an identity under $\pi$, the corresponding morphism
$\DD(\mu): i^*X \rightarrow j^*X$ is (co-)Cartesian. 

If $E$ is the trivial category, we omit $\pi$ from the notation, and talk about (absolutely) (co-)Cartesian objects. 
\end{DEF}

Note: If $\SSS$ is trivial or if $X$ lies over an object of the form $\pi^*S$ for $S \in \SSS(E)$ the notions $\pi$-coCartesian and $\pi$-Cartesian coincide. 

\begin{DEF}[cf.\@ {\cite[Definition~3.3.1]{Hor15}}] \label{DEFPROJECTOR}
Let $\DD \rightarrow \SSS$  and $\pi: I \rightarrow E$ be as in the previous Definition and let $S \in \SSS(I)$. If the
fully-faithful inclusion
\[ \DD(I)_{S}^{\pi-\cart} \hookrightarrow \DD(I)_{S} \qquad \DD(I)_{S}^{\pi-\cocart} \hookrightarrow \DD(I)_{S} \]
has a left (resp.\@ right) adjoint, we call that adjoint a {\bf left (resp.\@ right) (co)Cartesian projector}, denoted $\Box_!^\pi$ (resp.\@ $\Box_*^\pi$).
In this article it is always clear from the context, whether Cartesian or coCartesian objects are considered hence we will not use the notation $\blacksquare_!, \blacksquare_*$ from [loc.\@ cit.].
\end{DEF}

We want to extend (left or right) fibered multiderivators $\DD \rightarrow \SSS$ from a diagram category $\Dia'$ to a larger diagram category $\Dia$. 
Here $\SSS$ can be any pre-multiderivator (or even a 2-pre-multiderivator) satisfying (Der1) and (Der2). 
We assume that $\SSS$ is already defined on the larger diagram category $\Dia$ and that $\DD$ is a left (resp.\@ right) fibered multiderivator
over $\SSS$ such that also (FDer0 right)\footnote{at least when {\em neglecting} the multi-aspect} resp.\@ (FDer0 left) hold true. 

\begin{PAR}\label{PARAXIOMS}
Let $\Dia' \subset \Dia$ be diagram categories. 
We suppose given a functor
\[ N: \Dia \rightarrow \Dia'  \]
in which, forgetting 2-morphisms, $\Dia$ and $\Dia'$ are considered to be usual 1-categories,
together with a natural transformation
\[ \pi: N \Rightarrow \id \]
with the following properties:
\begin{enumerate}
\item[(N1)] For all $I \in \Dia$, the comma category $N(I) \times_{/I} N(I)$ (formed w.r.t.\@ the functors $\pi_I$) is in $\Dia'$ as well.
\item[(N2)] For all $I, J \in \Dia$, we have $N(I \coprod J) = N(I) \coprod N(J)$. Furthermore $N(\emptyset) = \emptyset$. 
\item[(N3)] For all $I \in \Dia$, $\pi_I$ is surjective on objects and morphisms, and has connected fibers.   
\item[(N4 left)] For any functor $\alpha: I \rightarrow J$ in $\Dia$ and for any object $j \in J$  the diagram
\[ \xymatrix{
N(I \times_{/J} j) \ar@{}[rd]|\Swarrow \ar[r] \ar[d] & N(I) \ar[d]^{ \alpha \circ \pi_I} \\
j  \ar[r] & J
} \]
is 2-Cartesian (i.e.\@ identifies the top left category with the corresponding comma category). 
 
\item[(N5 left)] For any pre-derivator $\DD$ satisfying (Der1) and (Der2) and for all $I \in \Dia$ with final object $i$ the functors
\[ \xymatrix{ \DD(N(I))^{\pi_I-\cart} \ar@<5pt>[rr]^-{n^*} & & \ar@<5pt>[ll]^-{p^*} \DD(\cdot) }  \]
form an adjunction, with $n^*$ left adjoint, where $n$ is some object with $\pi_I(n)=i$ and $p: N(I) \rightarrow \cdot$ is the projection. 
Furthermore the counit of the adjunction is the natural isomorphism and the unit  is an isomorphism on (absolutely) Cartesian objects. 
\end{enumerate}
\end{PAR}

Some immediate consequences of the axioms are listed in the following: 
\begin{LEMMA}[left]\label{BASICLEMMA}
\begin{enumerate}
\item Property (N5 left) is true w.r.t.\@ any object $n$ with $\pi_I(n)=i$.  
\item Let $\DD \rightarrow \SSS$ be a left fibered derivator satisfying also (FDer0 right)\footnote{at least when {\em neglecting} the multi-aspect}.
Let $J \in \Dia'$, let
$I \in \Dia$ with final object $i$, and let $S \in \SSS(I \times J)$ be an object. Let $n \in N(I)$ with $\pi_I(n)=i$. 
We 
denote
\begin{eqnarray*}
\pi_{I,J}:=(\pi_I \times \id_J) &:& N(I) \times J \rightarrow I \times J \\
i_J:=(i\times \id_J) &:& J \rightarrow I \times J \\
p_J :=(p \times \id_J)&:& N(I) \times J \rightarrow J \\
n_J:=(n \times \id_J)&:& J \rightarrow N(I) \times J \\
f:=\SSS(\nu)(S)&:& S \rightarrow p_J^* i_J^*S 
\end{eqnarray*}
where $\nu: \id_{I \times J} \Rightarrow i_J p_J$ is the obvious natural transformation.
The functors
\[ \xymatrix{ \DD(N(I) \times J)^{\pi_{I,J}-\cart}_{\pi_{I,J}^*S} \ar@<5pt>[rr]^-{n_J^*} & & \ar@<5pt>[ll]^-{(\pi_{I,J}^*f)^\bullet p_J^*} \DD(J)_{(i_J)^*S} }  \]
form an adjunction, with $n^*_J$ left adjoint. The counit is the natural isomorphism and the unit is an isomorphism on $\pr_2$-Cartesian objects. 

\item The natural morphism $n^*_J \rightarrow p_{J,!} (\pi_{I,J}^*f)_\bullet$ is an isomorphism on the subcategory $\DD(N(I)\times J)^{\pi_{I,J}-\cart}_{\pi_{I,J}^*S}$. 
\item The functor $(\pi_{I,J}^*f)^\bullet p_J^* n_J^*$ defines a left $p_J$-Cartesian projector (i.e.\@ a left adjoint to the inclusion)
\[ \DD(N(I) \times J )^{\pi_{I,J}-\cart}_{\pi_{I,J}^*S} \rightarrow \DD(N(I) \times J)^{p_J-\cart}_{\pi_{I,J}^*S}.  \]
\end{enumerate}
\end{LEMMA}
\begin{proof}
1. The fiber over $i$ in $N(I)$ is connected (N3). Hence on 
the subcategory $\DD(N(I))^{\pi_I-\cart}$ all functors $n^*$ for $n \in N(I)$ with $\pi_I(n)=i$ are isomorphic. 
Any of them can be thus taken as adjoint. 

2. The adjunction in question is the composition of the adjunctions
\[ \xymatrix{ \DD(N(I)\times J)^{\pi_{I,J}-\cart}_{\pi_{I,J}^*S} \ar@<5pt>[rr]^-{(\pi_{I,J}^*f)_\bullet } & & \ar@<5pt>[ll]^-{(\pi_{I,J}^*f)^\bullet } \DD(N(I)\times J)^{\pi_{I,J}-\cart}_{p_J^*i_J^*S} \ar@<5pt>[rr]^-{n_J^*} & & \ar@<5pt>[ll]^-{\pi_{I,J}^*} \DD(J)_{i_J^*S}. }  \]
For the second apply (N5 left) to the pre-derivator (fiber) $\DD_{(J,i_J^*S)}: K \mapsto \DD(K \times J)_{\pr_2^*i_J^*S}$.

3. Essentially uniqueness of adjoints. 

4. Follows from Lemma~\ref{LEMMAMONAD} applied to the monad $(\pi^*_{I,J}f_J)^\bullet p_J^* n^*_J$ associated with the adjunction
of 2. The assumptions are true because this monad has obviously values in 
absolutely Cartesian objects and by (N5 left) the unit is an isomorphism on absolutely Cartesian objects. 
\end{proof}

There are corresponding dual axioms (with a corresponding dual version of the Lemma which we leave to the reader to state):
\begin{enumerate}
\item[(N4 right)] For any functor $\alpha: I \rightarrow J$ in $\Dia$ and for any object $j \in J$  the diagram
\[ \xymatrix{
N(j \times_{/J} I) \ar@{}[rd]|\Nearrow \ar[r] \ar[d] & N(I) \ar[d] \\
j  \ar[r] & J
} \]
is 2-Cartesian (i.e.\@ identifies the top left category with the corresponding comma category). 
 
\item[(N5 right)] For any pre-derivator $\DD$ satisfying (Der1) and (Der2) with domain $\Dia'$ and for all $I \in \Dia$ with initial object $i$ the functors
\[ \xymatrix{ \DD(N(I))^{\pi_I-\cart} \ar@<5pt>[rr]^-{n^*} & & \ar@<5pt>[ll]^-{p^*} \DD(\cdot) }  \]
form an adjunction, with $n^*$ right adjoint, where $n$ is some object with $\pi_I(n)=i$ and $p: N(I) \rightarrow \cdot$ is the projection.
Furthermore the unit of the adjunction is the natural isomorphism and the counit is an isomorphism on (absolutely) Cartesian objects. 
\end{enumerate}

 \begin{PROP}\label{PROPCONSTRN}
 A strict functor $N$ as in \ref{PARAXIOMS} exists in the following cases and satisfies axioms (N1--3), (N4--5 left) (resp.\@ (N4--5 right)):
 \begin{equation*}
\begin{array}{c|c}  
\Dia' & \Dia \\
\hline
\hline
\Inv & \Cat \\
\Invpos & \Cat^\circ \\ 
\hline
\Dir & \Cat \\
\Dirpos & \Cat^\circ 
\end{array} 
\end{equation*}
 \end{PROP}

Here $\Cat^\circ$ is the category of those small categories in which identities do not factor nontrivially. Observe that $\Cat^\circ$ is self-dual, and that $\Dir \subset \Cat^\circ$ and $\Inv \subset \Cat^\circ$.
 
\begin{proof}
The functors $N$ are the following: 
For the pair $(\Dir \subset \Cat)$ denote by $\mathcal{N}^\circ(I)$ the semi-simplicial nerve of $I$. 
By applying the Grothendieck construction to the semi-simplicial set $\mathcal{N}^\circ(I)$ we obtain a directed diagram which is an opfibration with discrete fibers
over $(\Delta^{\circ})^{\op}$:
\[ N(I):=\int \mathcal{N}^\circ(I) \rightarrow (\Delta^{\circ})^{\op}. \]
It comes equipped with a natural functor $\pi_I: N(I) \rightarrow I$ mapping $(\Delta_n, i_0 \rightarrow \dots \rightarrow i_n)$ to $i_0$. 

For the pair $(\Dirpos \subset \Cat^\circ)$ denote by $\mathcal{N}^\circ(I)'$ the subobject of the semi-simplicial nerve of $I$ given by simplices
$\Delta_n  \rightarrow I$ in which no non-identity morphism is mapped to an identity. $N$ and $\pi$ are defined similarly and it is clear that
$N(I)$ is a directed poset. 

For the pair $(\Inv \subset \Cat)$,
by taking the opposite of the functor $N$ constructed for the pair $(\Dir \subset \Cat)$, we get an inverse diagram with a fibration to $\Delta^{\circ}$:
\[ N(I):= (\int \mathcal{N}^\circ(I))^{\op} \rightarrow \Delta^{\circ}. \]
It comes equipped with a natural functor $\pi_I: N(I) \rightarrow I$ mapping $(\Delta_n, i_0 \rightarrow \dots \rightarrow i_n)$ to $i_n$. 

For the pair $(\Invpos \subset \Cat^\circ)$ we have the obvious fourth construction. 

We have to check the axioms in each case, but will concentrate on the pairs $(\Dir \subset  \Cat)$ (in the following called case A) and $(\Dirpos \subset  \Cat^\circ)$ (in the following called case B), the others being dual. 

(N1--3) and (N4 right) are obvious. 

(N5 right)
Let $I \in \Dia$ be a diagram with initial object.  
We let $n$ (in both cases) be the object $(\Delta_0, i)$ of $N(I)$. 
The unit of the adjunction is the natural isomorphism
\[ u:\id \Rightarrow n^* p^*    \]
given by the equation $p \circ n = \id$.  

Recall from \cite[Lemma 2.3.3]{Hor15} (cf.\@ also \cite[Proposition 6.6]{Cis03}) the definition of the functor
\[ \xi: N(I) \rightarrow N(I) \]
which in case A is defined by
\[ (\Delta_n, i_0 \rightarrow \cdots \rightarrow i_n) \mapsto (\Delta_{n+1}, i \rightarrow i_0 \rightarrow \cdots \rightarrow i_n),   \]
and in case B by
\[ (\Delta_n, i_0 \rightarrow \cdots \rightarrow i_n) \mapsto \begin{cases}
 (\Delta_{n+1},  i \rightarrow i_0 \rightarrow \cdots \rightarrow i_n) & i_0 \not= i,  \\ 
 (\Delta_{n},  i_0 \rightarrow \cdots \rightarrow i_n) & i_0 = i.  
 \end{cases}  \]
There are (in both cases) natural transformations
\begin{equation}\label{eqn} \xymatrix{  \id_{N(I)} & \ar@{=>}[l] \xi \ar@{=>}[r] & n \circ p.  }  \end{equation}
%Applying $\DD$ and inserting something of the form $\pi^*S$ this yields
%\[ \xymatrix{ \pi^*S & \ar[l]  \xi^* \pi^*S \ar@{=}[r] &  p^* n^* \pi^* S } \]
%because $\pi \circ \xi = \pi \circ n \circ p$ is the constant functor to $I$ with value $i$. 
The counit of the adjuntion
\[ c: p^* n^* \Rightarrow \id \]
is given as follows: 
Applying $\DD$ to the sequence (\ref{eqn}) we get
\[ \xymatrix{  \id & \ar@{=>}[l] \xi^* \ar@{=>}[r] & (n \circ p)^*  }  \]
where the morphism to the right is obviously an isomorphism on $\pi_I$-Cartesian objects.
We let $c$ be the composition going from right to left. 

We will now check the counit/unit equations.

1. We have to show that the composition 
\begin{equation}\label{eqcu1} \xymatrix{ n^* \ar[r]^-{u n^*} & n^*p^*n^* \ar[r]^-{n^* c} & n^*  } \end{equation}
is the identity. Inserting the definitions, we get  that (\ref{eqcu1}) is $\DD$ applied to the following sequence of functors and natural transformations:
\[ \xymatrix{ n  & \ar[l]_-{e_0}  \xi n \ar[r]^-{e_1} & n p n  \ar@{=}[r] &  n  } \]
where $\xi n$ is the inclusion of $(\Delta_1, \id_i)$ in case A and is $n$ in case B. 
The morphisms $e_{0,1}$ are the (opposite of the) two inclusions $\Delta_0 \rightarrow \Delta_1$ in case A and the identity in case B. 
Hence in case B it is obvious that the composition (\ref{eqcu1}) is the identity while in case A 
it follows from Lemma~\ref{LEMMAPI1} (after applying $\DD$).

2. We have to show that the composition
\begin{equation}\label{eqcu2} \xymatrix{ p^* \ar[r]^-{p^* u} & p^*n^*p^* \ar[r]^-{c n^*} & p^*  } \end{equation}
is the identity. Inserting the definitions, we get that (\ref{eqcu2}) is $\DD$  applied to the following sequence of functors and natural transformations:
\[ \xymatrix{ p  & \ar@{=}[l]  p \xi \ar@{=}[r] & pnp  \ar@{=}[r] &  p  } \]
which consists only of identities. 
Hence the composition (\ref{eqcu2})  was the identity as well.
\end{proof}

\begin{LEMMA}[right]\label{LEMMAPI1}
Let $\DD$ be a pre-derivator with domain $\Dia'$, let $I$ be a diagram in $\Dia$ with initial object  $i$, and
let $\mathcal{E} \in \DD(\int \mathcal{N}^\circ (I))^{\pi_I-\cart}$. Then the two isomorphisms
\[ \xymatrix{
(\Delta_0, i)^*\mathcal{E} \ar@<5pt>@{<-}[rr]^-{\DD(e_0)} \ar@<-5pt>@{<-}[rr]_-{\DD(e_1)} && (\Delta_1, \id_i)^* \mathcal{E} 
} \]
in $\DD(\cdot)$ are equal. 
\end{LEMMA}
\begin{proof}
The underlying diagram of $\iota^* \mathcal{E}$, where $\iota: (\Delta^{\circ})^{\op} = \int \mathcal{N}^\circ (i) \rightarrow \int \mathcal{N}^\circ (I)$ is the inclusion, is a functor
\[ (\Delta^{\circ})^{\op} \rightarrow \DD(\cdot) \]
which maps all morphisms to isomorphisms. Since $\pi_1((\Delta^{\circ})^{\op})=1$, necessarily all parallel morphisms are mapped to the same isomorphism. 
\end{proof}

 \section{$\pi_I$-Cartesian projectors}
 
 Let $\Dia' \subset \Dia$ be diagram categories and
 let $N$ be a functor as in \ref{PARAXIOMS}. We also use the notation of Lemma~\ref{BASICLEMMA}.
 
 \begin{PROP}[left]\label{PROPCARTPROJ}
Let $\DD \rightarrow \SSS$ be a left fibered derivator satisfying also (FDer0 right)\footnote{at least when {\em neglecting} the multi-aspect} with domain $\Dia'$.
For all $I \in \Dia$, $J \in \Dia'$ and $S \in \SSS(I \times J)$ there exists a left $\pi_{I,J}$-Cartesian projector (cf.\@ Definition~\ref{DEFPROJECTOR})
\[ \Box_!^{\pi_{I,J}}: \DD(N(I) \times J )_{\pi_{I,J}^*S} \rightarrow \DD(N(I) \times J)_{\pi_{I,J}^*S}^{\pi_{I,J}-\cart}.  \] 
\end{PROP}
The proposition has an obvious right variant which we leave to the reader to formulate. 
\begin{proof}
Recall from \cite[6.8]{Hor15b} that a left fibered derivator with domain $\Dia'$ gives rise to a pseudo-functor of 2-categories (the multi-aspect is not needed here):
\begin{eqnarray*} \Psi: (\Dia')^{\cor}(\SSS) &\rightarrow& \mathcal{CAT} \\
 (I, S) &\mapsto& \DD(I)_S.
 \end{eqnarray*}

For each triple $(I, J, S)$ as in the statement, we define the following monad $T$ in $(\Dia')^{\cor}(\SSS)$. 
It has the properties that $\Psi(T)$ has values in $\DD(N(I) \times J)^{\pi_{I,J}-\cart}_{\pi_{I,J}^*S}$, and that the unit
$\id \Rightarrow \Psi(T)$ is an isomorphism on $\DD(N(I) \times J)^{\pi_{I,J}-\cart}_{\pi_{I,J}^*S}$. By Lemma~\ref{LEMMAMONAD} it follows that $\Psi(T)$ is a left $\pi_{I,J}$-Cartesian projector.

We have the 1-morphism $[\pi_{I,J}^{(S)}]$ in $\Dia^{\cor}(\SSS)$ and its  left adjoint $[\pi_{I,J}^{(S)}]'$, cf.\@ \cite[6.1--3 and Lemma~6.7]{Hor15b}. This adjunction defines a monad 
$T := [\pi_{I,J}^{(S)}] \circ [\pi_{I,J}^{(S)}]'$ on $(N(I) \times J, \pi_{I,J}^*S)$. Actually $T$ lies in $(\Dia')^{\cor}(\SSS)$ because of axiom (N1). 
Let us explicitly write down (a correspondence isomorphic to) $T$ as well as the unit: 
\[ \xymatrix{  & (N(I)\times_{/I}N(I)) \times J \ar[ld]_{(\pr_1,\id_J)} \ar[rd]^{(\pr_2, \id_J)} \\
N(I) \times J & & N(I) \times J  \\ 
& N(I) \times J \ar[uu]^{(\Delta_{12}, \id_J)} \ar[lu] \ar[ru] } \]
Here the topmost correspondence is equipped with the morphism $f: \pr_1^*\pi_{I,J}^*S \Rightarrow \pr_2^*\pi_{I,J}^*S$ induced by the natural transformation associated with the comma category.

Point-wise for $(n,j) \in N(I) \times J$ and $\mathcal{E} \in \DD(N(I) \times J)_{\pr_J^*S}$ we thus have for $i:=\pi_I(n)$:
\[ (n,j)^* (\Psi(T) \mathcal{E}) = p_{N(i \times_{/I} I),! } ((n,\id_{N(I)},j)^*f)_\bullet (\pr_1, j)^* \mathcal{E}.    \]
Obviously the right hand side depends only on $(i, j)$. Therefore the object $\Psi(T) \mathcal{E}$ is $\pi_{I,J}$-Cartesian.  
Note that by (N4 left) we have $i \times_{/I} N(I) = N(i \times_{/I} I)$.

The unit is given point-wise by the natural morphism
\[ \xymatrix{(n, n, j)^* (N(\pr_1), j)^* \mathcal{E}  \ar[r] & p_{N(i\times_{/I}I),! } ((n,\id_{N(I)},j)^*f)_\bullet (N(\pr_1), j)^* \mathcal{E}.  } \]
If $\mathcal{E}$ is $\pi_{I,J}$-Cartesian $(N(\pr_1), j)^* \mathcal{E}$ is $\pi_{I,J}$-Cartesian as well, and this map is an isomorphism by Lemma~\ref{BASICLEMMA}, 3.
%
%XXXX
%
%Hence we obtain a right $l$-(co)Cartesian  projector $(T_I^{\Dir})^\bullet$ and a left $r^{\op}$-(co)Cartesian projector $(T_I^{\Inv})_\bullet$
%provided  we can show the following (here for $(T_I^{\Dir})^\bullet$).
%The values of $(T_I^{\Dir})^\bullet$ are $l$-(co)Cartesian. And on $l$-(co)Cartesian the unit is an isomorphism. 
%
%Point-wise for $(n,j) \in N(I) \times J$ we have therefore 
%\[ n^* ((T_I^{\Dir})^\bullet \mathcal{E}) = \lim_{N^\circ(l(n)\times_{/I}I)} \pr_2^* \mathcal{E}    \]
%Obviously the right have side depend only on $l(n)$. 
%
%Now there is a commutative diagram 
%\[ \xymatrix{
% & \lim_{l(n) \times_{/I} N^\circ(I)} \mathcal{E} \ar[ld] \ar[d] \\
%n^* \mathcal{E} \ar[r] & (\Delta_0, l(n))^* \mathcal{E}
%} \]
%where the lower morphism is induced by the inclusion $\Delta_0 \rightarrow \Delta_n$, if $n= (\Delta_n, l(n)\rightarrow \dots \rightarrow i_n)$. 
%The property of being $l$-coCartesian therefore implies the the lower horizontal morphism is an isomorphism. 
%Since the right vertical morphism is an isomorphisms on $l$-(co)Cartesian objects (Lemma~\ref{}), also the diagonal morphism is an isomorphism on $l$-(co)Cartesian objects what was to be shown. 
\end{proof}

%Let $F: I \rightarrow \Cat$ be a pseudo functor with $i \in \Dir$ or $\Inv$? Then consider the morphism
%\[ \pi: \int N(F) \rightarrow \int F  \]
%Again, we get a monad on $(\int N(F), \pi^*S)$
%where $S \in \SSS(\int(F))$ is an object. 

%We want to show that it has values in $l$-Cartesian objects and that the unit is an iso on $l$-Cartesian objects. 

The following Lemma is well-known but due to lack of reference in this precise formulation we include its proof.

\begin{LEMMA}[left]\label{LEMMAMONAD}
Let $(\mathcal{C}, T, u, \mu)$ with $T: \mathcal{C} \rightarrow \mathcal{C}, u: \id \Rightarrow T$, and $\mu: T^2 \Rightarrow T$ be a monad in $\mathcal{CAT}$ and let $\mathcal{D} \subset \mathcal{C}$ be a full subcategory such that
\begin{enumerate}
\item $T$ takes values in $\mathcal{D}$,
\item the unit $u: \id \Rightarrow T$ is an isomorphism on objects of $\mathcal{D}$. 
\end{enumerate}
Then $T$, considered as functor $\mathcal{C} \rightarrow \mathcal{D}$, is left adjoint to the inclusion $\mathcal{D} \hookrightarrow \mathcal{C}$.
\end{LEMMA}
There is a corresponding right version in which a comonad gives rise to a right adjoint to the inclusion. 
\begin{proof}
%We focus on the monad case, the comonad case being dual. 
Consider $T$ as functor $\mathcal{C} \rightarrow \mathcal{D}$, which is possible by assumption 1.\@, and denote $\iota$ the inclusion $\mathcal{D} \hookrightarrow \mathcal{C}$. 
We define the unit  $\id \rightarrow \iota T$ of the adjunction to be the unit $u$ of the monad. 
The counit $T \iota \rightarrow \id$ is its inverse which exists by assumption 2. 

To show that this defines indeed an adjunction, we have to show the equation $uT=Tu$ as natural transformations $T \Rightarrow T^2$ (which is to say that the monad is an {\em idempotent} monad). 

By the definition of monad, we have the diagram
\[ \xymatrix{
T \ar@<-4pt>[r]_-{uT} \ar@<4pt>[r]^-{Tu} & T^2 \ar[r]^{\mu} & T 
} \]
in which both compositions are the identity. Hence to show that  $uT=Tu$ we have to show that one of them is an isomorphism, for 
then $\mu$ is an isomorphism as well, and hence after canceling $\mu$ we have $uT=Tu$. However $uT$ is an isomorphism by the assumptions.
\end{proof}

\begin{PROP}[left]\label{PROPCARTPROJ2}
Let $\DD \rightarrow \SSS$ be a left fibered derivator with domain $\Dia'$ satisfying also (FDer0 right)\footnote{at least when {\em neglecting} the multi-aspect}.
For each $I \in \Dia$, $J \in \Dia'$, and $S \in \SSS(I \times J)$ the functor 
\[ (N(\pr_1), \pr_2 \circ \pi_{I,J})^*:   \DD(N(I) \times J)_{\pi_{I,J}^*S}^{\pi_{I,J}-\cart} \rightarrow \DD(N(I \times J))_{\pi_{I \times J}^*S}^{\pi_{I \times J}-\cart}      \]
is an equivalence of categories. Its inverse is given by $(N(\pr_1), \pr_2 \circ \pi_{I,J})_!$ followed by the left $\pi_{I,J}$-Cartesian projector of Proposition~\ref{PROPCARTPROJ}. 
\end{PROP}
%
%\comment{
%\begin{PROP}
%The first induce equivalences
%\[ \pi_1^*:  \DD(\int N(F))^{\pi-\cart}_{\pi^*S} \rightarrow \DD(N(\int F))^{\pi-\cart}_{\pi^*S}  \]
%\end{PROP}
%}
\begin{proof}
Set $\pi_1:=(N(\pr_1),\pi_J N(\pr_2))$ and $\pi_2:=\pi_{I,J}$.
Consider the composition: 
\[ \xymatrix{  N( I \times J) \ar@/_20pt/[rrrrr]_{\pi_{I\times J}} \ar[rrr]^-{\pi_1} &&& N(I) \times J \ar[rr]^-{\pi_2} && I \times J  } \]

With the following notation
\[
\begin{array}{ll}
L_1 := [\pi_1^{(\pi_2^*S)}]' & R_1 := [\pi_1^{(\pi_2^*S)}] \\
L_2 := [\pi_2^{(S)}]' & R_2 := [\pi_2^{(S)}] 
\end{array}
\]
the two monads in $(\Dia')^{\cor}(\SSS)$ associated with $\pi_{I \times J}$ and $\pi_{I,J}$ are respectively: 
\begin{eqnarray*}
T_{I,J} &:=& R_2 \circ L_2,  \\
T_{I \times J} &:=& R_1 \circ R_2 \circ L_2 \circ L_1. 
\end{eqnarray*}

Consider the following diagram in which the objects are 1-morphisms in $(\Dia')^{\cor}(\SSS)$
and in which the 2-morphisms are given by the obvious units and counits:
\[ \xymatrix{
& R_1 \circ R_2 \circ L_2 \circ L_1 \circ R_1 \ar[r]^-{\numcirc{2}}  & R_1 \circ R_2 \circ L_2 \\
R_1 \ar@/_20pt/[rr]_{\id_{R_1}} \ar[ru]^(.4){\numcirc{1}} \ar[r] & R_1 \circ L_1 \circ R_1 \ar[u] \ar[r] & R_1 \ar[u]_{\numcirc{3}}
} \]
 
The composition of the second row is the identity. Note that this diagram lies actually in $(\Dia')^{\cor}(\SSS)$ although $L_2$ and $R_2$ do only lie in $\Dia^{\cor}(\SSS)$.

Hence after applying the functor $\Psi: (\Dia')^{\cor}(\SSS) \rightarrow \mathcal{CAT}$ and evaluating everything at a $\pi_{I \times J}$-Cartesian object
we obtain a diagram
in which $\numcirc{1}$ is mapped to an isomorphism because $T_{I \times J}=L_1 \circ L_2 \circ R_2 \circ R_1$ is mapped to a left $\pi_{I \times J}$-Cartesian projector. Also $\numcirc{3}$ is mapped to an isomorphism because $T_{I,J} = R_2 \circ L_2$ is mapped to a left $\pi_{I,J}$-Cartesian projector. Hence $\numcirc{2}$ is mapped to an isomorphism. However the image of $\numcirc{2}$ is $\Psi(R_1)=\pi_1^*$ applied to the unit
\[ \Psi(T_{I,J}) \pi_{1,!}\pi_1^* \leftarrow \id \]
That is hence an isomorphism.

The morphism 
\[ L_1 \circ L_2 \circ R_2 \circ R_1 \rightarrow \id \]
is mapped to an isomorphism on $\pi_{I \times J}$-Cartesian objects, hence the counit
\[ \pi_1^* \Psi(T_{I,J}) \pi_{1,!} \rightarrow \id \]
is an isomorphism on $\pi_{I \times J}$-Cartesian objects.
Therefore $\pi_1^*$ and $\Psi(T_{I,J}) \pi_{1,!}$ constitute an equivalence as claimed. 
\end{proof}

\section{Enlargement}

\begin{SATZ}\label{MAINTHEOREM}
Let $\Dia' \subset \Dia$ be two diagram categories.
Let $\DD \rightarrow \SSS$ be a left (resp.\@ right) fibered multiderivator with domain $\Dia'$ satisfying also (FDer0 right)\footnote{at least when {\em neglecting} the multi-aspect} (resp.\@ (FDer0 left)) such that $\SSS$ is defined on all of $\Dia$. 
Let $N: \Dia \rightarrow \Dia'$ be a functor as in \ref{PARAXIOMS} satisfying axioms (N1--3) and (N4--5 left) (resp.\@ (N4--5 right)).
Then 
\[ \EE(I)_S := \DD(N(I))_{\pi_I^*S}^{\pi_I-\cart} \]
defines a left (resp.\@ right) fibered multiderivator satisfying also (FDer0 right)\footnote{{\em neglecting} the multi-aspect} (resp.\@ (FDer0 left)) with domain $\Dia$. 
The restriction of $\EE$ to $\Dia'$ is canonically equivalent to $\DD$. Any other such enlargement of $\DD$ to $\Dia$ is equivalent to $\EE$. 

If $\Posf \subset \Dia'$ and {\em the fibers} of $\DD$ are in addition right (resp.\@ left) derivators with domain $\Posf$, so are the
fibers of $\EE$. 
\end{SATZ}
 
 \begin{proof}
 We begin by explaining the precise construction of $\EE \rightarrow \SSS$.
 The category $\EE(I)$ as a bifibration over $\SSS(I)$ is defined as the pull-back (cf.\@ \cite[2.23]{Hor15b})
 \[ \xymatrix{
\EE(I) \ar[r] \ar[d] & \DD(N(I))^{\pi_I-\cart} \ar[d] \\
\SSS(I) \ar[r]^-{\pi_I^*} & \SSS(N(I)).
} \]
Note that $\DD(N(I))^{\pi_I-\cart}$ over $\SSS(N(I))$ is not necessarily bifibered (the pull-back, resp.\@ push-forward functors will not
preserve the $\DD(N(I))^{\pi_I-\cart}$ subcategories, whereas the pull-back $\EE(I)$ is bifibered over $\SSS(I)$ by the following argument.
CoCartesian morphisms exist because for morphisms in the image of $\pi_I^*: \SSS(I) \rightarrow \SSS(N(I))$
the push-forward preserves the condition of being $\pi_I$-Cartesian. In the same way, 1-ary  Cartesian morphisms exist because 
for morphisms in the image of $\pi_I^*: \SSS(I) \rightarrow \SSS(N(I))$
the pull-back will preserve the condition of being $\pi_I$-Cartesian. For $n$-ary morphsims, $n\ge 2$ this need not to be true (the $n$-ary pull-backs are not necessarily ``computed point-wise'').
However, let $f \in \Hom_{\SSS(I)}(S_1, \dots, S_n; T)$ be a multimorphism. An adjoint of the push-forward
\[ (\pi_I^*f)_\bullet: \DD(N(I))_{\pi_I^*S_1}^{\pi_I-\cart} \times \dots \times \DD(N(I))_{\pi_I^*S_n}^{\pi_I-\cart} \rightarrow \DD(N(I))_{\pi_I^*T}^{\pi_I-\cart}  \]
w.r.t.\@ the $i$-th slot always exists, and is given by the usual pull-back $(\pi_I^*f)^{\bullet, i}$ followed by the right $\pi_I$-Cartesian projector of Proposition~\ref{PROPCARTPROJ}~(right). Since the right $\pi_I$-Cartesian projector exists only when $\DD \rightarrow \SSS$ is right fibered we can only show (FDer0 right) neglecting the multi-aspect if $\DD \rightarrow \SSS$ is {\em not} assumed to be right fibered as well. 
This shows that the pull-back $\EE(I) \rightarrow \SSS(I)$ is bifibered (with the mentioned restriction). 

A functor $\alpha: I \rightarrow J$ induces the following commutative diagram
\[ \xymatrix{
\DD(N(J))^{\pi_J-\cart} \ar[r]^{\alpha^*} \ar[d] & \DD(N(I))^{\pi_I-\cart} \ar[d] \\
\SSS(N(J)) \ar[r]^{\alpha^*} & \SSS(N(I))
} \]

Hence via pullback we get a diagram
\[ \xymatrix{
\EE(J) \ar[r]^{\alpha^*} \ar[d] & \EE(I) \ar[d] \\
\SSS(J) \ar[r]^{\alpha^*} & \SSS(I)
} \]
and the upper horizontal functor maps coCartesian morphism to coCartesian morphisms and Cartesian $1$-ary morphisms to Cartesian $1$-ary morphisms. 
This shows (FDer0 left) and the first part of (FDer0 right). The remaining part of (FDer0 right) will be shown in the end.

%To prove (FDer0 right) we have to check that it maps general $n$-ary Cartesian morphisms to Cartesian ones provided that $\alpha$ is an opfibration. 
%This will be proven by Lemma~\ref{} where it is used that all other axioms have been checked, in particular that $\EE \rightarrow \SSS$  is a right fibered derivator
%neglecting all multimorphisms for $n \not= 1$. 
  
We now construct the 2-functoriality of $\EE$ and concentrate on the left case, the other being dual.
A natural transformation $\mu: \alpha \Rightarrow \beta$ where $\alpha, \beta: I \rightarrow J$ are functors can be encoded by a functor
 \[ \mu: I \times \Delta_1 \rightarrow J \]
 such that $\mu_0$ (restriction to $I=I \times e_0$) is $\alpha$ and $\mu_1$ (restriction to $I=I \times e_1$) is $\beta$. 
 %Given an $\pi_J$-Cartesian object $\mathcal{E}$ over $S \in \SSS(J)$
 We use the equivalence
 \[ (N(\pr_1), \pr_2 \circ \pi_{N(I) \times \Delta_1})^*:  \DD(N(I) \times \Delta_1)^{\pi_{I,\Delta_1}-\cart}_{\pi_{I,\Delta_1}^*S } \iso \DD(N(I \times \Delta_1))^{\pi_{I \times \Delta_1}-\cart}_{\pi_{I \times \Delta_1}^*S}\]
  (cf.\@ Proposition~\ref{PROPCARTPROJ2}).
 From an object  $\mathcal{E} \in \EE(J)$ over $S \in \SSS(J)$ we get an object
\[ \Box^{\pi_{I,\Delta_1}}_! (N(\pr_1), \pr_2 \circ \pi_{N(I) \times \Delta_1})_! N(\mu)^* \mathcal{E} \]
which defines a morphism
\[ f_{\bullet} \alpha^* \mathcal{E} \rightarrow \beta^* \mathcal{E} \]
where $f$ is the composition 
\[ \xymatrix{
 N(\alpha)^* \pi_J^* S \ar[d]^\sim  && N(\beta)^* \pi_J^* S \ar[d]^\sim \\
\pi_J^* \alpha^*S  \ar[rr]^{(\pi_J^* \ast \SSS(\mu))(S)} && \pi_J^* \beta^*S
} \]
This defines the 2-functoriality. 

The axioms (Der1--2) for $\EE$ are clear (use axiom (N2) for (Der1)).

For the axioms (FDer3--4) we concentrate on the left case again, the other is dual. 

(FDer3 left) Let $\alpha: I \rightarrow J$ be a functor in $\Dia$. By assumption relative left Kan extensions exist for $\DD$, i.e.\@ the functor
\[ N(\alpha)^*:  \DD(N(J))_{\pi_J^*S} \rightarrow \DD(N(I))_{\pi_I^*\alpha^*S} \]
has a left adjoint $N(\alpha)_!$. Since by Proposition~\ref{PROPCARTPROJ} a left 
$\pi_J$-Cartesian projector $\Box_!^{\pi_J}$ exist on $\DD(N(J))_{\pi_J^*S}$, we obtain also a left adjoint to $N(\alpha)^*$ restricted to the respective subcategories, namely
\[ \Box_!^{\pi_J} N(\alpha)_!:  \DD(N(I))_{\pi_I^*\alpha^*S}^{\pi_I-\cart} \rightarrow \DD(N(J))_{\pi_J^*S}^{\pi_J-\cart}. \]
 
(FDer4 left) Consider a diagram as in the axiom: 
\[ \xymatrix{
I \times_{/J} j \ar[r]^-\iota \ar[d]_p \ar@{}[rd]|{\Swarrow} & I \ar[d]^\alpha \\
j \ar@{^{(}->}[r] & J
} \]
and the following induced diagram:
\[ \xymatrix{
N(I \times_{/J} j)  = N(I) \times_{/J} j \ar[rr]^-{N(\iota)} \ar[d] \ar@{}[rrd] && N(I) \ar[d]^{N(\alpha)} \\
N(J) \times_{/J} j \ar[rr]^-{\iota_2} \ar[d]_{p_{N(J \times_{/J} j)}} \ar@{}[rrd]|{\Swarrow^\mu} && N(J) \ar[d] \\
j \ar@{^{(}->}[rr] && J
} \]
By definition of the left $\pi_J$-Cartesian projector we have that 
\[ n^* \Box_!^{\pi_J}  \cong p_{N(J \times_{/J} j),!} (\SSS(\mu))_\bullet \iota_2^* \]
where $n$ is any element of $N(J)$ mapping to $j$. 
Therefore
\[ n^* \Box_!^{\pi_J} N(\alpha)_! \cong p_{N(I \times_{/J} j),!} (N(\alpha)^*(\SSS(\mu)))_\bullet N(\iota)^*.   \]
% 
% The adjoint to $\alpha^*$ for $\alpha: I \rightarrow J$ is given by pull-back along the correspondence
% \[ \xymatrix{  & N(J)\times_{/J}N(I) \ar[ld] \ar[rd] \\
%N(J) & & N(I) } \] 
%Composing this with any object $n \in N(J)$ with $l(n)=j$ we get that the value at $n$ is the pull-back along the correpondence
% \[ \xymatrix{  & N(j\times_{/J} I ) = j\times_{/J}N(I) \ar[ld] \ar[rd] \\
%j & & N(I) } \] 
% 
% 
Finally note that the composition
\[ \xymatrix{  \DD(N(I \times_{/J} j))^{\pi_{N(I \times_{/J} j)}-\cart}_{p^* j^* S }  \ar[rr]^-{\Box_!^{\pi} N(p)_!} & & \DD(N(j))_{\pi_j^*j^*S}^{\pi_j-\cart} \ar[rr]^-{\pi_{j,!}} & & \DD(j)_{j^*S}    } \]
is isomorphic to $p_{N(I \times_{/J} j),!}$ because left Cartesian projectors commute with relative left Kan extensions. 
By Lemma~\ref{BASICLEMMA} the second functor is an equivalence, and we deduce the isomorphism
\[ N(j)^* \Box_!^{\pi_J} N(\alpha)_! \cong \Box_!^{\pi_j} N(p)_! (N(\alpha)^*(\SSS(\mu)))_\bullet N(\iota)^*.   \]
A tedious check shows that this isomorphism can be identified with the base change morphism of (FDer4 left). 

(FDer5 right) In the right case of the Theorem (FDer0 left) has been established already. (FDer5 right) is the adjoint statement of the second part of (FDer0 left) \cite[Lemma~1.3.8]{Hor15}. 

(FDer5 left) By Lemma~\ref{LEMMAFDER5LEFT}, it suffices to prove (FDer5 left) for $p: I \rightarrow {\cdot}$, which means that the push-forward commutes with (homotopy) colimits in each variable.
That is, we have to see that the natural morphism
\[ \Box_! N(p)_! (N(p)^*\pi^*f)_\bullet(N(p)^*-, \dots, -, \dots, N(p)^*-) \rightarrow (\pi^*f)_\bullet (-, \dots, \Box_! N(p)_!, \dots, -)  \]
is an isomorphism. 
Now $\Box_!$ on $\DD(N(\cdot))$ is given by $\pi_! \pi^*$ for the projection $\pi: N(\cdot) \rightarrow \cdot$ (cf.\@ Lemma~\ref{BASICLEMMA}, 3--4.).
Therefore we may rewrite the morphism as
\[ \pi^* \pi_! N(p)_! (N(p)^*\pi^*f)_\bullet(N(p)^*-, \dots, -, \dots, N(p)^*-) \rightarrow (\pi^*f)_\bullet (-, \dots, \pi^* \pi_! N(p)_!,  \dots, -).  \]
Since all arguments, except the $i$-th one, are on $N(\cdot)$ and supposed to be $\pi$-Cartesian, they are in the essential image of $\pi^*$ as well ($\pi_! \pi^*$ is isomorphic to the identity on them, as just explained). 
Therefore it suffices to show (using FDer0 left) that
\[ \pi^* \pi_! N(p)_! (N(p)^*\pi^*f)_\bullet(N(p)^*\pi^*, \dots, -, \dots, N(p)^*\pi^*-) \rightarrow \pi^* (f_\bullet (-, \dots, \pi_! N(p)_!, \dots, -))  \]
is an isomorphism. 
This follows from (FDer5 left) for the original left fibered multiderivator $\DD \rightarrow \SSS$.

The remaining part of (FDer0 right): In the left case of the Theorem (FDer5 left) has been established already. The remaining part of (FDer0 right) is just the adjoint statement of (FDer5 left) \cite[Lemma~1.3.8]{Hor15}, hence it is satisfied automatically. 
In the right case of the Theorem,
by Lemma~\ref{LEMMAFDER0RIGHT}, it suffices to show (FDer0 right) for opfibrations of the form $\alpha: i \times_{/J} I \rightarrow I$. 
Axiom (N4 right) implies that also $N(\alpha): N(i \times_{/J} I) \rightarrow N(I)$ is an opfibration. 
%(alternative: in our explicit constructions of $N$ in the right case, all $N(\alpha)$ are opfibrations anyway). 
By Lemma~\ref{LEMMAPROJOPFIB}, $N(\alpha)^*$ commutes with the right Cartesian projectors as well. Therefore the statement is clear for opfibrations of this form.

That $\EE$ enlarges $\DD$ and that any other enlargement $\mathbb{F}$ is equivalent to $\EE$ is shown as follows. From Proposition~\ref{PROPCARTPROJ2} applied for $I= \cdot$  
and for $J \in \Dia'$ we get an equivalence
\[ \mathbb{D}(N(\cdot) \times J )^{\pi_{\cdot, J}-\cart}_{\pi_{\cdot, J}^*S} \cong \mathbb{D}(N(J))^{\pi_{J}-\cart}_{\pi_J^* S} \ \overset{\text{Def.}}{=}\ \EE(J)_S. \]
By (N5 left) applied to $I = \cdot$, and the derivator $I \mapsto \mathbb{D}(I \times J)_{\pr_2^*S}$ we get
 \[ \mathbb{D}(N(\cdot) \times J )^{\pi_{\cdot, J}-\cart}_{\pi_{\cdot, J}^*S} \cong \mathbb{D}(J)_{S}. \]
 All equivalences are compatible with pull-backs $\alpha^*$ and push-forwards, resp.\@ pull-backs along morphisms in $\SSS$.
 
With the same reasoning, setting $\Dia' = \Dia$, and $\mathbb{D}=\mathbb{F}$, we have for all $J \in \Dia$
\[ \mathbb{F}(J)_{S} \cong \mathbb{F}(N(J))^{\pi_{J}-\cart}_{\pi_J^* S}.  \]
Since $\mathbb{F}$ is equivalent to $\DD$ on $\Dia'$ we have 
\[ \ \mathbb{F}(N(J))^{\pi_{J}-\cart}_{\pi_J^* S} \cong \DD(N(J))^{\pi_{J}-\cart}_{\pi_J^* S} \overset{\text{Def.}}{=} \EE(J)_S.   \]

Finally, if the fibers of $\DD$ are right derivators with domain $\Posf$ (in the left case of the Theorem) then for all $S \in \SSS(K)$, with $K \in \Dia$, and for all functors $\alpha: I \rightarrow J$ in $\Posf$ the pull-back
\[ (\id \times \alpha)^* :  \DD(N(K) \times J)_{\pi_{K,J}^* \pr_1^*S} \rightarrow \DD(N(K) \times I)_{\pi_{K,I}^* \pr_1^* S} \]
has a right adjoint $(\id \times \alpha)_*$ such that Kan's formula holds true for it. It is easy to see that both $(\id \times \alpha)^*$  and $(\id \times \alpha)_*$  respect the subcategories of 
$\pi_{K,I}$-Cartesian, resp.\@ $\pi_{K,J}$-Cartesian objects. Therefore the fibers of $\EE$ are left derivators with domain $\Posf$ again, because
by Proposition~\ref{PROPCARTPROJ2} we have an equivalence
\[  \DD(N(K) \times J)^{\pi_{K,J}-\cart}_{\pi_{K,J}^* \pr_1^*S} \rightarrow \DD(N(K \times J))^{\pi_{K \times J}-\cart}_{\pi_{K \times J}^* \pr_1^*S} \ \overset{\text{Def.}}{=}\ \EE(K \times J)_{\pr_1^*S} \]
(via the pull-back).
\end{proof}

\begin{BEM}
The additional statement shows that if $\DD \rightarrow \SSS$ has stable, hence triangulated fibers 
(for this it is sufficient that the fibers are stable left and right derivators with domain $\Posf$) then also $\EE \rightarrow \SSS$ has stable, hence triangulated fibers. This allows to
establish, under additional conditions, that a left fibered multiderivator is automatically a right fibered multiderivator as well and vice versa (see \cite[\S 3.2]{Hor15}).
\end{BEM}
 
% Dual statement!

% Note that on the right hand side we could have written $\pi-\cocart$ as well, since all morphisms that have to be Cartesian/coCartesian lie over 
% identites in $\SSS(\cdot)$. 
 
% In the proof specify also the 1-/2-functoriality and the fibration...
 
 In the proof of Theorem~\ref{MAINTHEOREM} we used the following Lemmas: 
 
 \begin{LEMMA}\label{LEMMAFDER0RIGHT}
 %Let $\EE \rightarrow \SSS$ be a right fibered multiderivator except possibly for the second statement of (FDer0 right). Then $\EE \rightarrow \SSS$ is a right fibered multiderivator (i.e.\@ the second statement of (FDer0 right) holds), if for opfibrations of the form
The axiom (FDer0 right) in the definition of a right fibered multiderivator can be replaced by the following weaker axiom:
\begin{itemize}
\item[(FDer0 right')]
For each $I$ in $\Dia$ the morphism $p: \DD \rightarrow \SSS$ specializes to a fibered (multi)category and any functor of the form $\alpha: i \times_{/I} I \rightarrow I$ (note that this is an opfibration)
in $\Dia$ induces a diagram
\[ \xymatrix{
\DD(J) \ar[r]^{\alpha^*} \ar[d] & \DD(I) \ar[d]\\
\SSS(J) \ar[r]^{\alpha^*} & \SSS(I) 
}\]
of fibered (multi)categories, i.e.\@ the top horizontal functor maps Cartesian morphisms w.r.t.\@ the $i$-th slot to Cartesian morphisms w.r.t.\@ the $i$-th slot.
\end{itemize}
\end{LEMMA}
\begin{proof}
Let $\alpha$ now be an arbitrary opfibration. By axiom (Der2) it suffices to see that the natural morphism
\[ \alpha^* f^\bullet (-,\dots,-;-) \rightarrow (\alpha^* f)^\bullet (\alpha^* -,\dots,\alpha^* -;\alpha^* -) \]
is an isomorphism point-wise.

Using the homotopy Cartesian squares
\[ \xymatrix{
i \times_{/I} I \ar@{}[rd]|{\Nearrow^{\mu_i}} \ar[r]^-{\iota_I} \ar[d]_{p_i} &  I  \ar@{=}[d] \\
i \ar[r] & I 
} \quad \xymatrix{
j \times_{/J} J \ar@{}[rd]|{\Nearrow^{\mu_j}} \ar[r]^-{\iota_J} \ar[d]_{p_j} & J  \ar@{=}[d] \\
j \ar[r] &  J 
} \]
with $j = \alpha(i)$,
we have that $i^* \cong p_{i,*} \SSS(\mu_i)^\bullet \iota_I^*$ and $j^* \cong p_{j,*} \SSS(\mu_j)^\bullet \iota_J^*$. Using the assumption, i.e.\@ (FDer0 right'), the second statement of (FDer0 right) holds for $\iota_I$, and $\iota_J$, respectively. Thus, we are left to show that
\[ p_{j,*} \SSS(\mu_j)^\bullet (\iota_J^* f)^\bullet (\iota_J^*-,\dots,\iota_J^*-;\iota_J^*-) \rightarrow p_{i,*} \SSS(\mu_i)^\bullet (\iota_I^* f)^\bullet (\iota_I^*\alpha^* -,\dots,\iota_I^*\alpha^* -;\iota_I^*\alpha^* -) \]
is an isomorphism. We have $p_{I} = p_{J} \circ \rho$, where $\rho: i \times_{/I} I \rightarrow j \times_{/J} J$ is the functor induced by $\alpha$. Using (FDer5 right) we arrive at the morphism
\[ p_{j,*} \SSS(\mu_j)^\bullet (\iota_J^* f)^\bullet (\iota_J^*-,\dots,\iota_J^*-;\iota_J^*-) \rightarrow p_{j,*} \SSS(\mu_j)^\bullet (\iota_J^* f)^\bullet (\iota_J^* -,\dots,\iota_J^* -; \rho_* \rho^* \iota_J^* -) \]
induced by the unit $\id \rightarrow \rho_* \rho^*$. Since $\alpha$ is an opfibration, $\rho$ has a left adjoint $\rho'$ given as follows: It maps an object $j \rightarrow j'$ in $j \times_{/J} J$ to some (chosen for each such morphism) corresponding coCartesian morphism $i \rightarrow i'$. 
Hence $\rho_* = (\rho')^*$.
Since the unit 
\[ \id = \rho \circ \rho'   \]
is an equality the statement follows. 
 \end{proof}
 
 \begin{LEMMA}\label{LEMMAFDER5LEFT}
The axiom (FDer5 left) in the definition of a left fibered multiderivator can be replaced by the following weaker axiom.
 \begin{itemize}
 \item[(FDer5 left')] For any diagram $I \in \Dia$, and for any morphism $f \in \Hom(S_1, \dots, S_n; T)$ in $\SSS(\cdot)$ for some $n\ge 1$, the natural transformation of functors
\[ p_! (p^*f)_\bullet (p^*-, \cdots, p^*-,\ -\ , p^*-, \cdots, p^*-) \rightarrow  f_\bullet (-, \cdots, -,\ p_!-\ , -, \cdots, -), \]
where $p: I \rightarrow \cdot$  is the projection, is an isomorphism.
 \end{itemize}
  \end{LEMMA}
 \begin{proof}Let $\alpha: I \rightarrow J$ be an arbitrary opfibration. 
We have to show that the natural morphism
 \[ \alpha_! (\alpha^*f)_\bullet(\alpha^* -, \dots, -, \dots, \alpha^* -) \rightarrow f_\bullet( -, \dots, \alpha_! -, \dots,  -)    \]
 is an isomorphism. This can be proven point-wise by (Der2). Applying $j^*$ for an object $j \in J$, we arrive at 
 \[ p_{I_j, !} (p_{I_j}^*j^*f)_\bullet(p_{I_j}^* -, \dots, -, \dots, p_{I_j}^* -) \rightarrow (j^*f)_\bullet( -, \dots, p_{I_j,!} -, \dots,  -)    \]
 using (FDer0 left) and that $\alpha_!$ is computed fiber-wise for opfibrations. This is the statement of (FDer5 left'). 
 \end{proof}
 
 \begin{LEMMA}\label{LEMMAPROJOPFIB}
 For an opfibration of the form $\alpha: i \times_{/I} I \rightarrow I$ the pullback $N(\alpha)^*$ commutes with the right $\pi$-Cartesian projector, i.e.\@ the natural (exchange) morphism
 \[ N(\alpha)^* \Box^{\pi_I}_*  \rightarrow   \Box_*^{\pi_{i \times_{/I}I}} N(\alpha)^*  \]
 is an isomorphism. 
  \end{LEMMA}
\begin{proof}
Consider the following cube (where the objects in the rear face have been changed using (N4 right))
\begin{equation*}
 \xymatrix{
 & i \times_{/I} N(I) \times_{/I} N(I) \ar[rr]^{\pr_2^{i \times_{/I}I}} \ar@{}[ddrr]|(.65){\Nearrow^{\mu_{i \times_{I} I }}} \ar[dl]_{N(\alpha)\times_{/I}N(\alpha)} \ar[dd]^(.25){\pr_1^{i \times_{/I}I}} && i \times_{/I}  N(I) \ar[dd]^{} \ar[ld]_{N(\alpha)} \\ 
N(I) \times_{/I} N(I) \ar@{}[ddrr]|(.65){\Nearrow^{\mu_I}} \ar[rr]_(.4){\pr_2^{I}} \ar[dd]_{\pr_1^{I}} && N(I) \ar[dd]^(.65){} & \\
& i \times_{/I} N(I) \ar[rr]^(.35){N(\alpha)} \ar[dl]_(.4){N(\alpha)} && i \times_{/I} I \ar[dl]^{\alpha} \\
 N(I) \ar[rr]^{} &&  I }
\end{equation*}
We have by definition and (N4 right)
\[ \Box_*^{\pi_{i \times_{/I}I}} = \pr_{1,*}^{i \times_{/I}I} \SSS(\mu_{i \times_{/I}I})^\bullet (\pr_2^{i \times_{/I}I})^*, \]
and 
\[ \Box_*^{\pi_{I}} = \pr_{1,*}^I \SSS(\mu_I)^\bullet (\pr_2^I)^*. \]
Note that, in the cube, the rear face is just a pull-back of the front face. Since $N(\alpha)$ is an opfibration by (N4 right) we have therefore
\begin{eqnarray*}
  && N(\alpha)^* \pr_{1,*}^I \SSS(\mu_I)^\bullet (\pr_2^I)^*  \\
  &\cong& \pr_{1,*}^{i \times_{/I} I} (N(\alpha)\times_{/I}N(\alpha))^* \SSS(\mu_I)^\bullet (\pr_2^I)^*   \\
  &\cong& \pr_{1,*}^{i \times_{/I} I}  \SSS(\mu_{i \times_{/I} i})^\bullet (N(\alpha)\times_{/I}N(\alpha))^* (\pr_2^I)^*   \\
  &\cong& \pr_{1,*}^{i \times_{/I} I}  \SSS(\mu_{i \times_{/I} i})^\bullet  (\pr_2^{i \times_{/I} I})^* N(\alpha)^*.
\end{eqnarray*}
 \end{proof}
 
% We have $\Box_*^{\pi_I} = \pr_{I,1,*} \pr_{I,2}^* $ and  $\Box_*^{\pi_J} = \pr_{J,1,*} \pr_{J,2}^* $ and
% \[ N(\alpha)^* \pr_{J,1,*} \pr_{J,2}^* =  \pr_{J,1,*} N(\alpha)^* \pr_{J,2}^* = \pr_{J,1,*}  \pr_{J,2}^* N(\alpha)^*   \]
% because $N(\alpha)$ is an opfibration. (Attention: Is it really not needed here that we have a left fibered derivator as well. No it is not! Can be shown using the adjunctions...)
 %This kind of shows that right coCartesian projection ALWAYS commutes with $N(\alpha)^*$ since these are always opfibrations...

\begin{SATZ}\label{SATZCISINSKI}
Let $\mathcal{D} \rightarrow \mathcal{S}$ be a bifibration of multi-model categories, and let $\SSS$ be the represented pre-multiderivator of $\mathcal{S}$. 
For each small category $I$ denote by $\mathcal{W}_I$ the class of morphisms in $\Fun(I, \mathcal{D})$ which point-wise are weak equivalences in some fiber $\mathcal{D}_S$ (cf.\@ \cite[Definition 4.1.2]{Hor15}). 
The association
\[ I \mapsto \DD(I) := \Fun(I, \mathcal{D})[\mathcal{W}^{-1}_I]  \]
defines a left and right fibered multiderivator over $\SSS$ with domain $\Cat$. 
Furthermore the categories $\DD(I)$ are locally small. 
\end{SATZ}
\begin{proof}
In view of Theorem~\ref{MAINTHEOREM} it suffices to establish that we have equivalences of pre-multiderivators (compatible with the morphism to $\SSS$) 
\[ \DD \cong \EE^{\mathrm{left}}  \qquad  \DD \cong \EE^{\mathrm{right}} \]
where $\EE^{\mathrm{left}}$ (resp.\@ $\EE^{\mathrm{right}}$) is the left (resp.\@ right) fibered multiderivator --- the enlargement of $\DD$ --- constructed there w.r.t.\@ the $N$ given for the pair $(\Inv \subset \Cat)$ (resp.\@ $(\Dir \subset \Cat)$).

We sketch the left case, the other can be proven similarly. The idea goes back to Deligne in \cite[Expos\'e XVII, \S 2.4]{SGAIV3}, see also \cite[Proposition~4.1.9]{Hor15}.
The statement follows formally from the fact that, for all $S$, the localization of the fiber $\Fun(I, \mathcal{D})_{S}$ is isomorphic to the fiber $\EE(I)_S$ (Proposition~\ref{CENTRALPROP} below) and that the push-forward and pull-back functors can be derived
using left or right replacement functors, which exist by Lemma~\ref{LEMMAREPLACEMENT}.  

Let $I \in \Cat$ be a small category. 
We fix push-forward functors $f_\bullet^{\mathcal{D}}$, and $f_\bullet^{\EE}$, for the bifibrations $\Fun(I, \mathcal{D}) \rightarrow \SSS(I)$, and $\EE(I) \rightarrow \SSS(I)$, respectively. 
Note that it is not yet established that $\DD(I) = \Fun(I, \mathcal{D})[\mathcal{W}_I^{-1}] \rightarrow \SSS(I)$ is an opfibration.

We have a natural functor
\[ \Fun(I, \mathcal{D})[\mathcal{W}_I^{-1}] \rightarrow  \EE(I) \]
induced by $\pi_I^*$ and we will construct a functor going in the other direction
\[ \EE(I) \rightarrow \Fun(I, \mathcal{D})[\mathcal{W}^{-1}]. \]

By Proposition~\ref{CENTRALPROP} every object in $\EE(I)$ lies in the essential image of $\pi^*_I$, hence any morphism in $\EE(I)$ is of the form:
\[ \xi: \pi^*_I\mathcal{E}_1, \dots, \pi_I^*\mathcal{E}_n \rightarrow \pi_I^*\mathcal{F}  \]
lying over $\pi^*_If$ for some $f \in \Hom_{\SSS(I)}(S_1, \dots, S_n; T)$ --- or equivalently ---
\[ (\pi^*f)_\bullet^{\EE}(\pi^*\mathcal{E}_1, \dots, \pi^*\mathcal{E}_n) \rightarrow \pi^*\mathcal{F}  \]
in the fiber $\EE(I)_T$. It can be represented by a morphism in $\Fun(N(I), \mathcal{D})_{\pi^*T}$ of the form

\[ \pi^*_I (f_\bullet^{\mathcal{D}}(\widetilde{Q}\mathcal{E}_1, \dots, \widetilde{Q}\mathcal{E}_n)) \rightarrow \pi_I^*\mathcal{F}  \]
(in which all functors are underived functors).
Since the underived $\pi_I^*$ is fully-faithful by (N3)\footnote{This holds true for $\alpha^*: \Fun(J, \mathcal{D}) \rightarrow \Fun(I, \mathcal{D})$ for any category $\mathcal{D}$ and any functor $\alpha: I \rightarrow J$ which is surjective on objects and morphisms, and with connected fibers.}, this is the image under $\pi^*_I$ of a morphism
\[ \xi'': f_\bullet^{\mathcal{D}}(\widetilde{Q}\mathcal{E}_1, \dots, \widetilde{Q}\mathcal{E}_n) \rightarrow \mathcal{F}. \]
Proposition~\ref{CENTRALPROP} shows that $\xi''$ is well-defined in 
$\Fun(I, \mathcal{D})_{\pi^*T}[\mathcal{W}_{(I,T)}^{-1}]$ hence also well-defined in $\Fun(I, \mathcal{D})[\mathcal{W}_{I}^{-1}]$ because $\mathcal{W}_{(I,T)} \subset \mathcal{W}_{I}$. 

Equivalently $\xi''$ gives rise to a morphism
\[ \xi': \widetilde{Q}\mathcal{E}_1, \dots, \widetilde{Q}\mathcal{E}_n \rightarrow \mathcal{F} \]
over $f$, which composed with the formal inverses of the morphisms 
\[ \widetilde{Q}\mathcal{E}_i \rightarrow \mathcal{E}_i, \]
we define to be the image of $\xi$. A small check shows that this defines indeed a functor which is inverse to the one induced by $\pi^*_I$. 
\end{proof}

%TODO: Dies ist auch zu vergleichen mit Cisinskis Beweis, dass unten eine Modellstruktur existiert, falls die obere proper ist, oder so...

Let $\mathcal{D} \rightarrow \mathcal{S}$ be a bifibration of multi-model categories and let $\DD \rightarrow \SSS$ be the morphism of pre-multiderivators defined as in Theorem~\ref{SATZCISINSKI} (cf.\@ also \cite[Definition 4.1.2]{Hor15}), however, {\em with domain $\Inv$}. It is a left fibered multiderivator, satisfying also (FDer0 right), by \cite[Theorem 4.1.5]{Hor15}.

\begin{PROP}[left]\label{CENTRALPROP}
Let $I \in \Cat$ be a small category. 
Then $\pi_I^*$ induces an equivalence 
\[ \Fun(I, \mathcal{D})_{S}[\mathcal{W}^{-1}_{(I,S)}]  \cong \DD(N(I))^{\pi_I-\cart}_{\pi_I^*S} \]
where $\mathcal{W}^{-1}_{(I,S)}$ is the class of morphisms in $\Fun(I, \mathcal{D})_{S}$ which are point-wise in the corresponding $\mathcal{W}_{S_i}$ (weak equivalences in the
model structure on the fiber $\mathcal{D}_{S_i}$). 
\end{PROP}
There is an obvious right variant of the Proposition which we leave to the reader to state. Note, however, that for given $I$, the left hand side category is the same in both cases!
Note that $\DD(N(I))_{\pi_I^*S} = \Fun(N(I), \mathcal{D})_{\pi_I^*S}[\mathcal{W}^{-1}_{(N(I),\pi_I^*S)}]$
by \cite[Proposition 4.1.29]{Hor15}.
\begin{proof}
We have the (underived) adjunction
\[ \xymatrix{ \Fun(I, \mathcal{D})_{S} \ar@/^20pt/[r]^{\pi^*_I} &  \ar@/^20pt/[l]^{\pi_{I,!}^{(S)}} \Fun(N(I), \mathcal{D})_{\pi_I^*S}   } \]
with $\pi_{I,!}^{(S)}$ left adjoint. Both sides are equipped with classes classes $\mathcal{W}_{(I, S)}$, and $\mathcal{W}_{(N(I), \pi_I^*S)}$, respectively, of weak equivalences, and the right hand side is equipped even with the Reedy model category structure defined in 
\cite[4.1.18]{Hor15}. For the functors the following holds true:

\begin{enumerate}
\item
$\pi_I^*$ is exact (i.e.\@ respects the classes $\mathcal{W}_{(I, S)}$ and $\mathcal{W}_{(N(I), \pi_I^*S)}$). 

\item
$\pi_{I}^*$, when restricted to the localizations, has still a left adjoint defined by $\pi_{I,!} Q$, where $Q$ is the cofibrant resolution. 
Proof: It suffices to show that $\pi_{I,!} Q$ defines a absolute left derived functor. For this it suffices to see that $\pi_{I,!}$ maps weak equivalences
between cofibrant objects to weak equivalences. This can be checked point-wise. Consider the 2-Cartesian diagram
\begin{equation}\label{eqdiacomma} \vcenter{ \xymatrix{
 N(I) \times_{/I} i \ar[r]^-{\iota} \ar[d]^p \ar@{}[rd]|{\Swarrow^\mu} & N(I) \ar[d]^{\pi_I} \\
i \ar[r] & I
} } \end{equation}
We have the following isomorphism between underived functors
\[ i^* \pi_{I,!}  \cong p_! \SSS(\mu)_\bullet \iota^* . \]
Now, $\iota^*$ preserves cofibrant objects (w.r.t.\@ the model structure considered in \cite[4.1.18]{Hor15}) by \cite[Lemma 4.1.27]{Hor15}
and we know that $\SSS(\mu)_\bullet$ and $p_!$ both map weak equivalences between cofibrant objects to weak equivalences (both are left Quillen). Therefore $\pi_{I,!}$ maps weak equivalences between cofibrant objects to weak equivalences as well. 
\end{enumerate}

We have to show that the unit (between the derived functors)
\[  \id \rightarrow \pi_I^* \pi_{I,!} \]
is an isomorphism on $\pi_I$-Cartesian objects. This can be shown after applying $n^*$ for any $n \in N(I)$. Set $i:=\pi_I(n)$. We get:
\[  n^* \rightarrow i^* \pi_{I,!} \]
which is the same as 
\[ (n')^* \iota^* \rightarrow p_{!} \SSS(\mu)_\bullet \iota^* \]
where $n'=(n, \id_i) \in N(I) \times_{/I}i$. However this is an isomorphism by Lemma~\ref{BASICLEMMA}, 3.

We have to show that the counit (between the derived functors)
\[ \pi_{I,!} \pi_I^* \rightarrow \id \]
is an isomorphism. Again, it suffices to see this point-wise. After applying $i^*$ and using the above we arrive at the morphism
\[ p_{!} \SSS(\mu)_\bullet \iota^* \pi_I^* \rightarrow i^* \]
for $p$ and $\iota$ as in 2.

On $\pi_I$-Cartesian objects, $p_{N(I) \times_{/I} i,!}$ is equal to the evaluation at any $n$ mapping to $i$ by Lemma~\ref{BASICLEMMA}, 3.
\end{proof}

Recall that for a functor $F: \mathcal{C}_1 \times \cdots \times \mathcal{C}_n \rightarrow \mathcal{D}$, and for given classes of ``weak equivalences'' $\mathcal{W}_{\mathcal{C},1}, \dots, \mathcal{W}_{\mathcal{C},n}, \mathcal{W}_{\mathcal{D}}$ we call a {\em left replacement functor 
adapted to $F$} a collection of endofunctors $Q_i: \mathcal{C}_i \rightarrow \mathcal{C}_i$  with natural transformations $Q_i \Rightarrow \id_{\mathcal{C}_i}$ consisting point-wise
of weak equivalences such that $F \circ (Q_1, \dots, Q_n)$ maps weak equivalences to weak equivalences. It follows that $F \circ (Q_1, \dots, Q_n)$
is an absolute left derived functor of $F$. Similarly for the right case. 

\begin{LEMMA}\label{LEMMAREPLACEMENT}
Let $(N, \pi)$ be the functor and natural transformation constructed in Proposition~\ref{PROPCONSTRN} for the pair $(\Inv \subset \Cat)$ (a left case) 
and $(\widetilde{N}, \widetilde{\pi})$ the ones for the pair $(\Dir \subset \Cat)$ (a right case).
Let $I \in \Cat$ be a small category and let $f \in \Hom_{\SSS(I)}(S_1, \dots, S_n; T)$ be a multimorphism. 
On $\prod_i \Fun(I, \mathcal{D})_{S_i}$ the functor $\widetilde{Q}:= \pi_{I,!}^{(S_i)} Q \pi_I^*$, where $Q$ is the cofibrant replacement in the Reedy model category $\Fun(N(I), \mathcal{D})_{\pi_I^* S_i}$ \cite[4.1.18]{Hor15}, is a left replacement functor adapted to $f_\bullet$ by virtue of the composition
\[  \pi_{I,!}^{(S_i)} Q \pi_I^* \rightarrow  \pi_{I,!}^{(S_i)} \pi_I^* \rightarrow \id. \]
In particular $f_\bullet$ has a total left derived functor. 
Similarly the functor $\widetilde{R}:=\widetilde{\pi}_{I,*}^{(T)} R \widetilde{\pi}_I^*$, where $R$ is the fibrant replacement in the Reedy model category $\Fun(\widetilde{N}(I), \mathcal{D})_{\widetilde{\pi}^*_I T}$ (the opposite of \cite[4.1.18]{Hor15}), is a right replacement functor adapted to $f^{\bullet, i}$ (in the covariant argument) by virtue of the composition
\[ \widetilde{\pi}_{I,*}^{(T)} R \widetilde{\pi}_I^* \leftarrow  \widetilde{\pi}_{I,*}^{(T)} \widetilde{\pi}_I^* \leftarrow \id. \]
More precisely, the functor
\[ f^{\bullet,i}( \widetilde{Q}^{\op}, \overset{\widehat{i}}{\dots}, \widetilde{Q}^{\op}; \widetilde{R}) \]
maps weak equivalences to weak equivalences. 
In particular $f^{\bullet,i}$ as a functor
\[ \Fun(I, \mathcal{D})_{S_1}^{\op} \times \overset{\widehat{i}}{\cdots} \times \Fun(I, \mathcal{D})_{S_n}^{\op} \times \Fun(I, \mathcal{D})_{T} \rightarrow \Fun(I, \mathcal{D})_{S_i}  \]
has a total right derived functor. 
The so constructed derived functors form an adjunction in $n$ variables again. 
\end{LEMMA}
\begin{proof}
As usual, we omit the bases from the relative Kan extension functors, they are clear from the context.
Let $i \in I$ be an object. Note that 
\[  i^* \pi_{I,!} Q \pi_I^* \cong p_* \SSS(\mu)^\bullet \iota^* Q \pi_I^* \] 
(Notation as in (\ref{eqdiacomma})) and we have seen on the proof of Proposition~\ref{CENTRALPROP} that $p_*$,  $\SSS(\mu)^\bullet$, and $\iota^*$ preserve cofibrations and weak equivalences between cofibrations. 
Therefore $\pi_{I,!} Q \pi^*_I$ has image in point-wise cofibrant objects. 
$f_\bullet$ maps point-wise weak equivalences between point-wise cofibrant objects to point-wise weak equivalences. 
%Now
%\[ f_\bullet (\pi_{I,!} Q \pi_{i}^* , \dots, \pi_{I,!} Q \pi_{I}^*) = \pi_{I,!} (\pi_I^*f)_\bullet (\pi_I^* \pi_{I,!} Q \pi_I^*, \dots, Q \pi^*_I, \dots, \pi_I^* \pi_{I,!} Q \pi_{I}^*)  \]
%by the underived projection formula. However $(\pi^*f)_\bullet (\pi_I^* \pi_{I,!} Q \pi_*, \dots, Q \pi^*_I, \dots, \pi^*_I \pi_{I,!} Q \pi_{I,*})$ maps obviously weak equivalences to weak equivalences (all arguments are point-wise cofibrant). Furthermore it maps still to cofibrant objects in the Reedy model category structure of \cite[4.1.18]{Hor15} on $\Fun(N(I), \mathcal{D})_{\pi_I^* S_i}$ because $Q \pi^*_I$ does. Therefore $\pi_{I,!}$ of it maps still weak equivalences to weak equvalences. 

We have
\[ f^{\bullet, j} (\pi_{I,!} Q \pi^*_I , \dots, \pi_{I,!} Q  \pi_I^*; \widetilde{\pi}_{I,*} R \widetilde{\pi}_I^*) \cong \widetilde{\pi}_{I,*} (\widetilde{\pi}_I^* f)^{\bullet,j} (\widetilde{\pi}_I^* \pi_{I,!} Q \pi_I*, \dots, \widetilde{\pi}^*_I \pi_{I,!} Q \pi_I^*;  R \widetilde{\pi}_I^*)  \]
and $(\widetilde{\pi}_I^* f)^{\bullet,j} (\widetilde{\pi}_I^* \pi_{I,!} Q \pi_I^*, \dots, \widetilde{\pi}^*_I \pi_{I,!} Q \pi_I^*;  R \widetilde{\pi}_I^*)$ maps weak equivalences to weak equivalences and maps to fibrant objects in the Reedy model category structure (opposite to \cite[4.1.18]{Hor15}) because $(\widetilde{\pi}_I^* f)^{\bullet,j}$ is part of a Quillen adjunction in $n$ variables and cofibrations are the point-wise ones in that model-category structure. 
Therefore $f^{\bullet, j} (\pi_{I,!} Q \pi^*_I , \dots, \pi_{I,!} Q  \pi_I^*; \widetilde{\pi}_{I,*} R \widetilde{\pi}_I^*)$ maps weak equivalences to weak equivalences. 
%
%
%For $f$ 1-ary we have
%\[ f_\bullet \pi_! Q \pi_* = \pi_! (\pi^*f)_\bullet Q \pi^*  \]
%and $(\pi^*f)_\bullet$, being left Quillen, maps cofibrant objects to cofibrant objects and weak equivaneces between them to weak equivalences, and for $\pi_!$ this was shown as well, hence
%the whole expression maps weak equivalences to weak equivalences.  
%
%WHAT ABOUT $n$-ary morphisms. 
\end{proof}

\begin{BEM}
In the non-fibered case, Cisinski \cite[Th\'eor\`eme 6.17]{Cis03} shows that for a right proper model category $\mathcal{D}$ where the cofibrations are the monomorphisms, a similar construction
like  in the Lemma may even be used to construct a model category structure on $\Fun(I, \mathcal{D})$ itself in which the weak equivalences are the point-wise ones. Probably a similar statement is true in the fibered situation, but we have not checked this.  
\end{BEM}

\commentempty{
\begin{PROP}
Let $\DD \rightarrow \SSS$ be a right fibered multiderivator (sat. FDer0 left as well) with domain $\Dir$. 
If $I$ has an initial object $i$, we have that the extraction of $(\Delta_0, i)$ 
\[ \DD(N(I))_{\pi_I^*S}^{\cocart} \leftarrow \DD(\cdot) \]
is an equivalence.
\end{PROP}
\begin{proof}
We first prove that the fully-faithful inclusion
\[ \DD(N(I))_{\pi_I^*S}^{\pi_I-\cart} \leftarrow \DD(N(I))_{\pi_I^*S}^{\cocart} \]
has a right adjoint $\Box_*$.

Finally look at: 
\[ \xymatrix{
\DD(N(I))^{\cocart}_{\pi_I^*S} \ar[r]^-{(\Delta_0, i)^*} & \DD(\cdot)_{S(i)} \ar[r]^-{f_\bullet \pi_I^*}& \DD(N(I))_{\pi_I^*S}^{\cart}  
} \]
By construction the composition is the right coCartesian projector, hence an equivalence. The other composition
\[ \xymatrix{
\DD(\cdot)_{S(i)} \ar[r]^-{f_\bullet \pi_I^*} & \DD(N(I))^{\cocart}_{\pi_I^*S}   \ar[r]^-{(\Delta_0, i)^*} &  \DD(\cdot)_{S(i)}
} \]
is an isomorphism for trivial reasons. This establishes the equivalence.  
\end{proof}

\begin{KOR}
If $I$ is a diagram with initial object, we have an isomorphism of functors
\[ \pi_* f^\bullet = (\Delta_0, i)^* : \DD(N(I))^{\pi_I-\cart}_{\pi_I^*S} \rightarrow \DD(\cdot)_{S(i)} \]
\end{KOR}
\begin{proof}
We have the diagram
\[ \xymatrix{
 \DD(\cdot)_{S(i)} \ar[r]^{f_\bullet \pi_I^*} & \DD(N(I))^{\cocart}_{\pi_I^*S} \ar[r]^{\mathrm{incl}}   &  \DD(N(I))^{\pi_I-\cart}_{\pi^*S} 
} \]
and we know that the left functor is an equivalence and the right adjoint of the right functor is given by $(\Delta_0, i)^*$ followed by $f_\bullet \pi_I^*$.

\end{proof}

}

\begin{PROP}\label{PROPLEFTRIGHT}
Let $\Dia' \subset \Dia$ be two diagram categories and $N$, and $\widetilde{N}$, be two functors as in \ref{PARAXIOMS} satisfying (N1--3) and (N4--5 left), (resp.\@ (N4--5 right)). 
And suppose, in addition, that for each $I \in \Dia$ the diagram $N(I) \times_{/I} \widetilde{N}(I)$ is in $\Dia'$ as well. Let $\DD \rightarrow \SSS$ be a left and right fibered multiderivator such that $\SSS$ extends to all of $\Dia$. 
Then we have an equivalence of categories
\[  \DD(N(I))^{\pi_I-\cart}_{\pi_I^*S} \cong \DD(\widetilde{N}(I))^{\widetilde{\pi}_I-\cart}_{\widetilde{\pi}_I^*S}  \]
compatible with pull-back along functors $\alpha: I \rightarrow J$ and, for all morphisms $f$ in $\SSS(I)$, intertwining push-forward along $\pi_I^*f$ with that along $\widetilde{\pi}_I^*f$. 
\end{PROP}
\begin{proof}
Consider the adjunction 
\begin{equation}\label{eqadj}
 \vcenter{ \xymatrix{ \DD(N(I))_{\pi_I^*S} \ar@/^20pt/[rr]^{\pr_{2,!} \SSS(\mu)_\bullet \pr_1^*} && \ar@/^20pt/[ll]^{\pr_{1,*} \SSS(\mu)^\bullet \pr_2^*}  \DD(\widetilde{N}(I))_{\widetilde{\pi}_I^*S} }  } 
  \end{equation}
induced by the following 2-commutative diagram:
\[ \xymatrix{
N(I) \times_{/I} \widetilde{N}(I) \ar@{}[rd]|{\Nearrow^\mu} \ar[r]^-{\pr_2} \ar[d]_{\pr_1} & \widetilde{N}(I) \ar[d]^{\widetilde{\pi}_I} \\
N(I)  \ar[r]_{\pi_I} & I
} \]
It suffices to show that the unit (resp.\@ the counit) of the adjunction are isomorphisms on $\pi_I$-Cartesian (resp.\@ $\widetilde{\pi}_I$-Cartesian) objects. 
We concentrate on the counit (the unit case is analogous) and show that 
\[ c: \pr_{2,!} \SSS(\mu)_\bullet \pr_1^* \pr_{1,*} \SSS(\mu)^\bullet \pr_2^* \rightarrow \id \]
is an isomorphism on $\widetilde{\pi}_I$-Cartesian objects. It suffices to see this after pull-back to any $\widetilde{n}$. Hence we have to show that
\[\widetilde{n}^* c:  \widetilde{n}^* \pr_{2,!} \SSS(\mu)_\bullet \pr_1^* \pr_{1,*} \SSS(\mu)^\bullet \pr_2^* \rightarrow \widetilde{n}^*  \]
is an isomorphism. 
Consider also an object $n \in N(I)$ mapping to the same $i \in I$ as $\widetilde{n}$. They give rise to an element $\kappa = (n, \widetilde{n}, \id_i) \in  N(I) \times_{/I} \widetilde{N}(I)$ and we may rewrite the morphism as
\[ \kappa^* \pr_2^* \pr_{2,!} \SSS(\mu)_\bullet \pr_1^* \pr_{1,*} \SSS(\mu)^\bullet \pr_2^* \rightarrow \kappa^* \pr_2^*. \]
Consider the following commutative diagram of functors and natural transformations (all given by units and counits of the obvious adjunctions):
\[ \xymatrix{
&&  \kappa^* \pr_2^*  \\
\kappa^* \pr_2^* \pr_{2,!} \SSS(\mu)_\bullet \pr_1^* \pr_{1,*} \SSS(\mu)^\bullet \pr_2^*   \ar[r]  \ar[rru]^{\kappa^* \pr_2^* c} & \kappa^* \pr_2^* \pr_{2,!}  \SSS(\mu)_\bullet \SSS(\mu)^\bullet  \pr_2^*  \ar[r] & \kappa^* \pr_2^* \pr_{2,!}   \pr_2^*  \ar[u] \\
\kappa^* \SSS(\mu)_\bullet \pr_1^* \pr_{1,*} \SSS(\mu)^\bullet \pr_2^* \ar[u]^{\numcirc{1}} \ar[r] \ar@/_20pt/[rr]_{\numcirc{2}} & \kappa^* \SSS(\mu)_\bullet  \SSS(\mu)^\bullet \pr_2^*  \ar[u]  \ar[r] & \kappa^*  \pr_2^*  \ar@/_50pt/[uu]_{\id}  \ar[u]
} \]
It suffices to verify that $\numcirc{1}$ and $\numcirc{2}$ are both isomorphisms on $\widetilde{\pi}_I$-Cartesian objects. 

$\numcirc{1}$: It is clear that both functors in the adjunction (\ref{eqadj}) have image in $\pi_I$-Cartesian (resp.\@ $\widetilde{\pi}_I$-Cartesian) objects. Hence is suffices to see that 
\[ \kappa^*  \SSS(\mu)_\bullet \pr_1^*  \rightarrow \kappa^* \pr_2^* \pr_{2,!} \SSS(\mu)_\bullet \pr_1^*  \]
is an isomorphism on $\widetilde{\pi}_I$-Cartesian objects. This is the same as the natural morphism
\[ n^*  \rightarrow p_! \SSS(\mu)_\bullet \iota^*  \]
induced by the 2-commutative diagram
\[ \xymatrix{
N(I\times_{/I} i)  \ar@{}[rd]|{\Nearrow^\mu} \ar[r]^-{p} \ar[d]_{\iota} & i \ar[d]^{} \\
N(I)  \ar[r]_{\pi_I} & I
} \]
and this is an isomorphism by Lemma~\ref{BASICLEMMA}, 3 (left case).

$\numcirc{2}$: We may rewrite $\numcirc{2}$ as:
\[ n^* \pr_{1,*} \SSS(\mu)^\bullet \pr_2^* \rightarrow \widetilde{n}^* \]
and have to show that
it is an isomorphism on $\widetilde{\pi}_I$-Cartesian objects. This is the same as the natural morphism
\[ p_* \SSS(\mu)_\bullet \iota^*   \rightarrow  \widetilde{n}^* \]
for the 2-commutative diagram
\[ \xymatrix{
 \widetilde{N}(i \times_{/I} I) \ar@{}[rd]|{\Nearrow^\mu} \ar[r]^-{\iota} \ar[d]_{p} & \widetilde{N}(I) \ar[d]^{\widetilde{\pi}_I} \\
i  \ar[r] & I
} \]

and this is an isomorphism by Lemma~\ref{BASICLEMMA}, 3 (right case).
\end{proof}

\newpage
\bibliographystyle{abbrvnat}
\bibliography{6fu}

\end{document}